\documentclass[a4paper,12pt]{article}
\usepackage[centertags]{amsmath}
\usepackage{amsfonts}
\usepackage{amssymb}
\usepackage{amsthm}
\usepackage{dsfont}
\usepackage{tikz, subfigure}
\usepackage{parskip}

\newtheorem{theorem}{Theorem}[section]
\newtheorem{lemma}{Lemma}[section]
\newtheorem{cor}{Corollary}[section]

\newtheorem{conjecture}{Conjecture}

\begin{document}

\title{Counting $\beta$-expansions and the absolute continuity of Bernoulli Convolutions}
\author{Tom Kempton}
\maketitle

\begin{abstract}\noindent We study the typical growth rate of the number of words of length $n$ which can be extended to $\beta$-expansions of $x$. In the general case we give a lower bound for the growth rate, while in the case that the Bernoulli convolution associated to parameter $\beta$ is absolutely continuous we are able to give the growth rate precisely. This gives new necessary and sufficient conditions for the absolute continuity of Bernoulli convolutions.\end{abstract}
\section{Introduction}
Given $\beta \in (1,2]$ and $x\in I_{\beta}:=\left[0,\frac{1}{\beta-1}\right]$, a $\beta$-expansion of $x$ is a sequence $(x_n)_{n=1}^{\infty}\in\{0,1\}^{\mathbb N}$ for which
\[
x=\sum_{n=1}^{\infty}x_i\beta^{-i}.
\]
For $\beta=2$ these are the familiar binary expansions, almost every $x\in[0,1]$ has a unique binary expansion but there are a countable number of $x$ which have two expansions. The situation for $\beta\in(1,2)$ is much more complicated; the set $\mathcal E_{\beta}(x)$ of $\beta$-expansions of $x\in I_{\beta}$ may be a singleton set \cite{EJH}, or have any positive integer cardinality \cite{ErdosJooMultiplicity}, but it is typically uncountable \cite{EJH, Sidorov1}. Since the work of Renyi \cite{Renyi} and Parry \cite{Parry}, there has been a great deal of interest in expansions of real numbers in non-integer bases.


Finer information on the structure of $\mathcal E_{\beta}(x)$ can be given by defining 
\[
\mathcal E^n_{\beta}(x):=\{(x_1,\cdots,x_n) \in \{0,1\}^n|\exists (x_{n+1},x_{n+2},\cdots) : x=\sum_{k=1}^{\infty}x_k\beta^{-k}\}
\]
and studying the growth of $\mathcal N_n(x;\beta):=|\mathcal E^n_{\beta}(x)|$ as $n$ increases. The maximal growth rate of $\mathcal N_n(x;\beta)$ was studied in \cite{FurstenbergKatznelson}, while the growth rate for Lebesgue almost-every $x$ was studied in \cite{FengSidorov}, where precise information was given in the case that $\beta$ is a Pisot number as well as some estimates for the general case.

This question is also linked to the question of absolute continuity of Bernoulli convolutions. For $\beta\in(1,2)$, the Bernoulli convolution $\nu_{\beta}$ is defined as the weak star limit of the measures
\[
\nu_{\beta,n}:=\frac{1}{2^n}\sum_{a_1\cdots a_n\in\{0,1\}^n} \delta_{\sum_{i=1}^n a_i\beta^{-i}},
\]
where $\delta_x$ denotes the Dirac measure supported on $\{x\}$. It is a long standing open problem to determine for which parameters $\beta$ the corresponding Bernoulli convolution is absolutely continuous; it is known that almost every $\beta\in(1,2)$ admits an absolutely continuous Bernoulli convolution \cite{SolomyakAC}, but that if $\beta$ is a Pisot number then the corresponding convolution is singular \cite{ErdosPisot}. For a survey of Bernoulli convolutions, including various alternative definitions, see \cite{SolomyakSixty}. 

It was shown in \cite{PollicottWeissDimensions} that the quantity
\[
\limsup_{n\to\infty}\left(\frac{\beta}{2}\right)^n\mathcal N_n(x;\beta)
\]
is uniformly bounded if and only if $\nu_{\beta}$ is absolutely continuous and has bounded density. This generalised a previous condition of Garsia and was termed the Garsia-Erd\H{o}s condition. 

The main focus of this article is to prove results on the typical growth rate of $\mathcal N_n(x;\beta)$ in the case that the Bernoulli convolution associated to $\beta$ is absolutely continuous. A consequence of our results is that we are able to give new necessary and sufficient conditions for the absolute continuity of Bernoulli convolutions without imposing any further conditions on the density. 

We also show how the growth rate of $\mathcal N_n(x;\beta)$ can be studied using the ergodic properties of the random $\beta$-transformation of Dajani and Kraaikamp \cite{DKRandom}. Our techniques allow us to prove an almost everywhere lower bound on the growth rate of $\mathcal N_n(x;\beta)$ for $\beta\in(1,2)$, extending a result of \cite{FengSidorov}.

In section \ref{Clustering} we give our results for the case that $\nu_{\beta}$ is absolutely continuous and explain how these can be viewed as a generalisation of the techniques of Garsia. In section \ref{Operators} we prove these results using an operator based on the self-similarity of the invariant density. In section \ref{RB} we give a result for general $\beta$ using the random $\beta$-transformation.


\section{Results}\label{Clustering}

We find it more convenient to deal with a renormalised version of $\mathcal N_n(x;\beta)$. For $n\in\mathbb N$ we define functions $f_n:I_{\beta}\to\mathbb R$ by
\[
f_n(x):=\left(\frac{(\beta-1)\beta^{n}}{2^n}\right)\mathcal N_n(x;\beta).
\]
Since $\sum_{i=n+1}^{\infty}x_i\beta^{-i}$ can take any value in $[0,\frac{1}{(\beta-1)\beta^n}]$, we have that a word $x_1,\cdots, x_n$ is an element of $\mathcal E^n_{\beta}(x)$ if and only if
\[
\sum_{i=1}^{n}x_i\beta^{-i}\in\left[x-\frac{1}{(\beta-1)\beta^{n}},x\right].
\]

So for each $n$, the value of $f_n(x)$ is equal to $\frac{(\beta-1)\beta^{n}}{2^n}$ multiplied by the number of words $x_1,\cdots, x_n$ for which $\sum_{i=1}^nx_i\beta^{-i}\in \left[x-\frac{1}{(\beta-1)\beta^{n}},x\right]$. Then since there are $2^n$ choices of $x_1,\cdots, x_n$, each of which cover an interval of size $\frac{1}{(\beta-1)\beta^{n}}$, we see that
\[
\int_{I_{\beta}}f_n(x) dx=1.
\]
We further define
\[
\overline f(x):= \limsup_{n\to\infty} f_n(x)
\]
and
\[
\underline f(x):=\liminf_{n\to\infty} f_n(x).
\]
The following is our main theorem.


\begin{theorem}\label{thm1}
The Bernoulli convolution $\nu_{\beta}$ is absolutely continuous if and only if 
\[
0<\int_{I_{\beta}}\overline f(x)dx<\infty. 
\]
In this case, one automatically has that $\int_{I_{\beta}}\overline f(x)dx \leq 2$ and that if $\nu_{\beta}$ has density $h_{\beta}$ then 
\[h_{\beta}(x)=\frac{\overline f(x)}{\int_{I_{\beta}}\overline f(x)dx}.\] 
\end{theorem}

We have a slightly weaker theorem for $\underline f$.

\begin{theorem}\label{thm2}
Suppose that 
\[
0<\int_{I_{\beta}}\underline f(x)dx<\infty. 
\]
Then $\nu_{\beta}$ is absolutely continuous with density function \[h_{\beta}(x)=\frac{\underline f(x)}{\int_{I_{\beta}}\underline f(x)dx}.\]
Conversely, if $\nu_{\beta}$ is absolutely continuous with bounded density function $h_{\beta}$ then $\underline f(x)$ satisfies
\[
0<\int_{I_{\beta}}\underline f(x)dx<\infty. 
\]
\end{theorem}


Theorems \ref{thm1} and \ref{thm2} give immediate consequences for the growth rate of $\mathcal N_n(x;\beta)$. In particular, because $\nu_{\beta}$ is absolutely continuous for almost every $\beta\in(1,2)$ and has bounded density for almost every $\beta\in(1,\sqrt{2})$ (see \cite{SolomyakAC}) we have the following.
\begin{cor}
For almost every $\beta\in(1,2)$, and for (Lebesgue) almost every $x\in I_{\beta}$,
\[
\limsup_{n\to\infty} \frac{1}{n}\log(\mathcal N_n(x;\beta))=\log\left(\frac{2}{\beta}\right)
\]
For almost every $\beta\in(1,\sqrt{2})$ and for (Lebesgue) almost every $x\in I_{\beta}$,
\[
\lim_{n\to\infty} \frac{1}{n}\log(\mathcal N_n(x;\beta))=\log\left(\frac{2}{\beta}\right).
\]
\end{cor}

These results contrast with the case that $\beta$ is a Pisot number. It was shown in \cite{FengSidorov} that if $\beta$ is a Pisot number then  
\[
\lim_{n\to\infty} \frac{1}{n}\log(\mathcal N_n(x;\beta))<\log\left(\frac{2}{\beta}\right).
\]

\subsection{Comparison with the Garsia entropy technique}

In \cite{GarsiaEntropy}, Garsia approached the question of the potential singularity of Bernoulli convolutions by considering the entropy of the measures $\nu_{\beta,n}$, where the entropy of a discrete measure $m$ supported on a finite set $\{x_1,\cdots ,x_k\}$ is defined as
\[
H(m)=-\sum_{i=1}^k m(x_i)\log(m(x_i)).
\]
Garsia showed that if
\[
\lim_{n\to\infty}\frac{H(\nu_{\beta,n})}{n}<\log(\beta)
\]
then $\nu_{\beta}$ is singular. This was used to provide an alternative proof that Pisot numbers admit singular Bernoulli convolutions, but has not been successful in determining whether $\nu_{\beta}$ is singular or absolutely continuous for any other $\beta$. Garsia's approach can be characterised as quantifying to what extent the finite sums $\sum_{i=1}^n x_i\beta^{-i}$ coincide for different choices of $x_1,\cdots, x_n$. 

Rather than asking for coincidence of these finite sums, our approach asks about clustering. The functions $f_n$ have integral $1$ and take values $f_n(x)$ proportional to the number of elements of
\[
\mathcal D_n:=\left\{\sum_{i=1}^n x_i\beta^{-i}:x_i\in\{0,1\}\right\}
\]
which lie in $\left[x-\frac{1}{(\beta-1)\beta^{n}},x\right]$. Then the regularity of the function $f_n$ describes the degree of clustering in $\mathcal D_n$. In particular, if $\int_{I_{\beta}}\overline f(x)dx\in\{0,\infty\}$ then this corresponds to a high degree of clustering in the sets $\mathcal D_n$ as $n$ tends to infinity. 

Thus, our results can be interpreted as saying that the absolute continuity or singularity of the Bernoulli convolution can be determined by measuring the degree of clustering in the sets $\mathcal D_n$.\footnote{This perspective on the question of counting $\beta$-expansions arose out of conversations with Evgeny Verbitskiy, many thanks to Evgeny for these.}


\section{An Operator on Densities}\label{Operators}
The results stated in the previous section come from simple analysis of a naturally defined operator related to Bernoulli convolutions. Bernoulli convolutions satisfy the self-similarity relation
\[
\nu_{\beta}=\frac{1}{2}\left(\nu_{\beta}\circ T_0 + \nu_{\beta}\circ T_1\right)
\]
where $T_i(x):\beta x-i$. This implies that if $\nu_{\beta}$ is absolutely continuous with density $h_{\beta}$ then $h_{\beta}$ must also satisfy a self-similarity relation:
\[
h_{\beta}(x)=\frac{\beta}{2}\left(h_{\beta}(\beta x)+h_{\beta}(\beta x-1)\right).
\]
We define the operator $P:L^1(\mathbb R)\to L^1(\mathbb R)$ by
\[
P(f)(x)=\frac{\beta}{2}\left(f(\beta x)+f(\beta x-1)\right).
\]
$P$ preserves the set of densities 
\[
D:=\{f\in L^1: f\geq 0, \int_{\mathbb R} f(x)dx=1, x\notin I_{\beta}\implies f(x)=0\}.
\]
The absolute continuity of $\nu_{\beta}$  is equivalent to the existance of a function $h_{\beta}\in D$ satisfying $P(h_{\beta})=h_{\beta}$.

\begin{lemma}\label{BasicP} The following are elementary properties of $P$
\begin{enumerate}
 \item $f\leq g\implies P(f)\leq P(g)$
 \item $P(kf)=kP(f)$ $\forall k\in\mathbb R$
 \item If $f\geq 0$ then $\int_{\mathbb R}P(f)=\int_{\mathbb R} f$
 \item If $P$ has fixed point $h_{\beta}$ then $||P(f)-h_{\beta}||_1\leq ||f-h_{\beta}||_1$ for all $f\in L^1$.
\end{enumerate}

\end{lemma}
\begin{proof}
The first two statements follow immediately from the definition of $P$. The third statement follows from the fact that, for each $i$
\begin{equation*}\label{int}
\int_{\mathbb R} f(T_i(x))dx= \frac{1}{\beta} \int_{\mathbb R}f(x)dx.
\end{equation*}
The fourth statement is just an application of the triangle inequality. Using that $P(h_{\beta})=h_{\beta}$, we have that
\begin{eqnarray*}
\int_{\mathbb R}|P(f)(x)-h_{\beta}(x)|dx&=&\int_{\mathbb R}\left|\frac{\beta}{2}\left(f(\beta x)-h_{\beta}(\beta x)+f(\beta x-1)-h_{\beta}(\beta x-1)\right)\right|dx\\
&\leq& \frac{\beta}{2}\left(\int_{\mathbb R}\left|f(\beta x)-h_{\beta}(\beta x)\right|dx+\int_{\mathbb R}\left|f(\beta x-1)-h_{\beta}(\beta x-1)\right|dx\right)\\
&=&\frac{\beta}{2}\left(\frac{1}{\beta}\int_{\mathbb R} |f(x)-h_{\beta}(x)|dx+\frac{1}{\beta}\int_{\mathbb R} |f(x)-h_{\beta}(x)|dx\right)\\
&=&\int_{\mathbb R} |f(x)-h_{\beta}(x)|dx
\end{eqnarray*}
as required.
\end{proof}
We now link $P$ with the functions defined in the first section.
\begin{lemma}\label{fP}
\[
f_n=P^n((\beta-1)\chi_{I_{\beta}})
\]
where $\chi_{I_{\beta}}$ is the indicator function on the interval $I_{\beta}$.
\end{lemma}
\begin{proof}
Expanding $P$, we see that
\[
P^n((\beta-1)\chi_{I_{\beta}})(x)=(\beta-1)\left(\frac{\beta}{2}\right)^n\sum_{x_1,\cdots, x_n\in\{0,1\}^n}\chi_{I_{\beta}}(T_{x_n}\circ T_{x_{n-1}}\circ\cdots T_{x_1}(x))
\]
But a word $x_1,\cdots, x_n\in\{0,1\}^n$ is in $\mathcal E_{\beta}^n(x)$ if and only if $T_{x_n}\circ T_{x_{n-1}}\circ\cdots T_{x_1}(x)\in I_{\beta}$, see for example \cite{DKRandom}. This gives that
\[
\sum_{x_1,\cdots, x_n\in\{0,1\}^n}\chi_{I_{\beta}}(T_{x_n}\circ T_{x_{n-1}}\circ\cdots T_{x_1}(x))=\mathcal N_n(x;\beta).
\]
Then since $f_n:=(\beta-1)\left(\frac{\beta}{2}\right)^n\mathcal N_n(x;\beta)$ we have the required result.
\end{proof}

The next four lemmas prove our main results.
\begin{lemma}\label{overlinebound}
If $\nu_{\beta}$ is absolutely continuous then $0<\int_{I_{\beta}}\overline f(x)dx$.
\end{lemma}

\begin{proof}
We suppose that $\nu_{\beta}$ is absolutely continuous with density $h_{\beta}$, and for a contradiction suppose that $\int_{I_{\beta}}\overline f(x)dx=0$. Consequently we have that $f_n=P^n((\beta-1)\chi_{I_{\beta}})\to 0$ for Lebesgue almost every $x$. Now if $P^n((\beta-1)\chi_{I_{\beta}})\to 0$ almost everywhere, then linearity and monotonicity of $P$ give that for any bounded $g\in D$, $P^n(g)\to 0$  almost everywhere. Since $h_{\beta}\in L^1$ we can, for any $0<\epsilon<\frac{1}{2}$, take a bounded function $g_{\epsilon}\in D$ with 
\[
||h_{\beta}-g_{\epsilon}||_1<\epsilon.
\]
But $P^n(g_{\epsilon})\to 0$ a.e. since $g_{\epsilon}$ is bounded. Then since $h_{\beta}$ has integral $1$, we see that eventually
\[
||h_{\beta}-P^n(g_{\epsilon})||_1>\frac{1}{2}>\epsilon.
\]
This contradicts part 4 of lemma \ref{BasicP}, the non expansiveness of $P$ in $L^1$.
\end{proof}
\begin{lemma}\label{overlineupperbound}
If $\nu_{\beta}$ is absolutely continuous then $\int_{I_{\beta}}\overline f(x)dx\leq2$.
\end{lemma}
\begin{proof}
In order to show that $\int_{I_{\beta}}\overline f(x)dx\leq 2$ we in fact show that $\overline f(x)\leq 2h_{\beta}(x)$.
Given some word $x_1,\cdots, x_n\in\{0,1\}^n$, we see that all sequences $\underline x$ starting with word $x_1,\cdots, x_n$ are $\beta$-expansions of points
\[
x=\sum_{i=1}^{\infty}x_i\beta^{-i}\in\left[\sum_{i=1}^n x_i\beta^{-i},\sum_{i=1}^n x_i\beta^{-i}+\frac{1}{(\beta-1)\beta^n}\right].
\]

Now we suppose for each $j\in\{1,\cdots,\mathcal N_n(x;\beta)\}$, $x_1^j,\cdots, x_n^j\in\mathcal E_{\beta}^n(x)$. Then 
\[
\sum_{i=1}^nx_i^j\beta^{-i}\in\left[x-\frac{1}{(\beta-1)\beta^n},x\right],
\]
and so for any $m>n$, $j\in\{1,\cdots,\mathcal N_n(x;\beta)\}$ and any word $x_1,\cdots, x_m$ starting with $x_1^j,\cdots, x_n^j$ we have that
\[
\sum_{i=1}^{m} x_i\beta^{-i}\in\left[x-\frac{1}{(\beta-1)\beta^n},x+\frac{1}{(\beta-1)\beta^n}\right].
\]
There are at least $2^{m-n}\mathcal N_n(x;\beta)$ such words. Hence
\[
\nu_{\beta,m}\left[x-\frac{1}{(\beta-1)\beta^n},x+\frac{1}{(\beta-1)\beta^n}\right]\geq \frac{1}{2^m}(\mathcal N_n(x;\beta).2^{m-n})=\frac{\mathcal N_n(x;\beta)}{2^n}
\]
Then using that $\nu_{\beta,m}\to\nu_{\beta}$, multiplying each side by $\beta^n$, and using the fact that 
\[
\lim_{n\to\infty}\frac{(\beta-1)\beta^n}{2}\nu_{\beta}\left[x-\frac{1}{(\beta-1)\beta^n},x+\frac{1}{(\beta-1)\beta^n}\right]=h_{\beta}(x)
\]
we get that
\[
\overline f(x)=\limsup_{n\to\infty} (\beta-1)\left(\frac{\beta}{2}\right)^n\mathcal N_n(x;\beta)\leq 2h_{\beta}(x). 
\]
\end{proof}

A similar argument was used in Appendix C of \cite{PollicottWeissDimensions} to bound $\left(\frac{\beta}{2}\right)^n\mathcal N_n(x;\beta)$ in the case that $h_{\beta}(x)$ is bounded. The authors attributed the argument to Yuval Peres. We now relate $\overline f$ to $h_{\beta}$.

\begin{lemma}\label{Pinvariant}
If $\overline f$ has positive finite integral then $\nu_{\beta}$ is absolutely continuous and $\frac{\overline f}{\int_{I_{\beta}}\overline f(x)dx} = h_{\beta}$. The same is true for $\underline f$.
\end{lemma}
 
\begin{proof}
We proved in lemma \ref{fP} that $f_n(x)=P^n((\beta-1)\chi_{I_{\beta}})$, which gives in particular that
\[
f_n(x)=P(f_{n-1}(x))=\frac{\beta}{2}\left(f_{n-1}(\beta x)+f_{n-1}(\beta x-1)\right)
\]
Then since $\limsup (a_n+b_n)\leq \limsup a_n + \limsup b_n$, we have that
\[
\overline f (x)\leq \frac{\beta}{2}(\overline f (\beta x)+ \overline f(\beta x-1)),
\]
i.e. $P(\overline f)\geq \overline f$. But since $\overline f$ is non negative, part 3 of lemma \ref{BasicP} gives that $\int_{\mathbb R}P(\overline f)=\int_{\mathbb R}\overline f$. So we must have that $P(\overline f)=\overline f$ almost everywhere.

Then we see that $\dfrac{\overline f}{\int_{\mathbb R}\overline f}$ has integral one and is a fixed point of $P$, and therefore it must be equal to $h_{\beta}$ almost everywhere. Similar arguments work for $\underline f$ using that $P(\underline f)\leq \underline f$.
\end{proof}

Lemmas \ref{overlinebound}, \ref{overlineupperbound} and \ref{Pinvariant} complete the proof of Theorem \ref{thm1}.
Finally, we prove the following lemma.

\begin{lemma}\label{underlinebound}
If $\nu_{\beta}$ is absolutely continuous and has bounded density function $h_{\beta}$ then $0<\int_{I_{\beta}}\underline f(x)dx\leq1.$
\end{lemma}

\begin{proof}
Fatou's lemma gives us that
\[
\int_{I_{\beta}}\underline f(x)dx\leq \liminf_{n\to\infty} \int_{I_{\beta}}f_n(x)dx=1
\]

Now we assume that $\nu_{\beta}$ is absolutely continuous with bounded density $h_{\beta}<C$ for some $C>0$. We suppose that $h_{\beta}(x)>0$, this holds for almost every $x\in I_{\beta}$, see \cite{MauldinSimon}. Applying the operator $P$ $n$ times to $h_{\beta}$ we see that
\begin{eqnarray*}
h_{\beta}(x)&=&\left(\frac{\beta}{2}\right)^n \sum_{x_1,\cdots, x_n\in \mathcal E_{\beta}^n(x)}h_{\beta}(T_{x_n}\circ T_{x_{n-1}}\circ\cdots T_{x_1}(x))\\
&\leq & \left(\frac{\beta}{2}\right)^n \mathcal N_n(x;\beta).C,
\end{eqnarray*}
where we have used that $|\mathcal E_{\beta}^n(x)|=N_n(x;\beta)$ and that $h_{\beta}(T_{x_n}\circ T_{x_{n-1}}\circ\cdots T_{x_1}(x))<C$ for each $x_1,\cdots, x_n\in\{0,1\}^{n}$.

Then rearranging and taking the $\liminf$ we see that 
\[
\liminf_{n\to\infty} P^n(x)=\liminf_{n\to\infty}(\beta-1)\left(\frac{\beta}{2}\right)^n\mathcal N_n(x;\beta)\geq \frac{(\beta-1)h(x)}{C}
\]
as required.
\end{proof}
Lemmas \ref{Pinvariant} and \ref{underlinebound} complete the proof of Theorem \ref{thm2}.

\begin{section}{Counting expansions using the random $\beta$-transformation}\label{RB}
In section \ref{Clustering} we gave results on the growth rate of $\mathcal N_n(x;\beta)$ which hold for almost all $\beta\in(1,2)$. In this section we  take an alternative approach, using the ergodic theory of the random $\beta$-transformation to study the growth rate of $\mathcal N_n(x;\beta)$ for all $\beta\in(1,2)$. 

In \cite{FengSidorov}, Feng and Sidorov proved that for all $\beta\in\left(1,\frac{1+\sqrt{5}}{2}\right)$ there exists a constant $c(\beta)>0$ such that for all $x\in I_{\beta}$
\[
\liminf_{n\rightarrow\infty}\dfrac{\log (\mathcal N_n(x;\beta))}{n}\geq c(\beta).
\]

We extend this result beyond the case that $\beta\in(1,\frac{1+\sqrt{5}}{2})$.

\begin{theorem}\label{betathm}
For every $\beta$ in $(1,2)$, there exists $c(\beta)>0$ such that for almost all $x \in (0,\frac{1}{\beta-1}),$
\[
\liminf_{n\rightarrow\infty}\dfrac{\log (\mathcal N_n(x;\beta))}{n}\geq c(\beta).
\]
\end{theorem}
In fact the above theorem extends to all non integer $\beta>1$, the proof for $\beta>2$ follows that given below but has more complicated notation, and can be found in the author's thesis. We stress that,  in the case that $\beta\geq\frac{1+\sqrt{5}}{2}$, the almost all $x\in(0,\frac{1}{\beta-1})$ of Theorem \ref{betathm} cannot be extended to hold for all $x\in(0,\frac{1}{\beta-1})$. There are, for example, points $x$ with unique $\beta$-expansion.

In \cite{DKRandom}, Dajani and Kraaikamp defined the random $\beta$-transformation  $K_{\beta}:\{0,1\}^{\mathbb N}\times[0,\frac{1}{\beta-1}]\rightarrow \{0,1\}^{\mathbb N}\times[0,\frac{1}{\beta-1}]$ by
\[
K_{\beta}(\omega,x)=\left\lbrace\begin{array}{c c}(\omega,T_0(x))& x \in [0,\frac{1}{\beta})\\ (\sigma(\omega),T_{\omega_1}(x))& x \in [\frac{1}{\beta},\frac{1}{\beta(\beta-1)}]\\ (\omega,T_1(x)) & x \in (\frac{1}{\beta(\beta-1)},\frac{1}{\beta-1}]\end{array}\right. .
\]
Given a pair $(\omega,x)$, a beta expansion of $x$ is generated by iterating $K_{\beta}(\omega,x)$. If the $nth$ iteration of $K_{\beta}(\omega,x)$ applies $T_0$ to the second coordinate we put $x_n=0$, if it applies $T_1$ to the second coordinate we put $x_n=1$. The sequence $(x_n)$ is a $\beta$-expansion of $x$ \cite{DKRandom}.

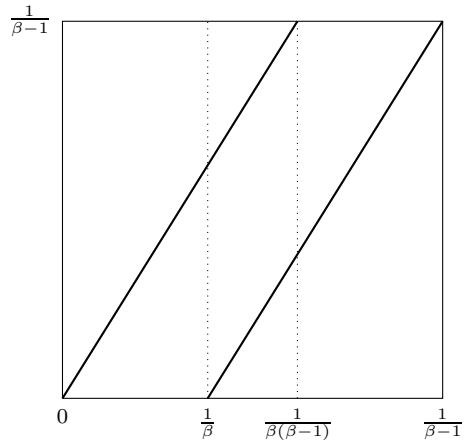
\begin{figure}[h]
\centering
\begin{tikzpicture}[scale=5]
\draw(0,0)node[below]{\scriptsize 0}--(.382,0)node[below]{\scriptsize$\frac{1}{\beta}$}--(.618,0)node[below]{\scriptsize$\frac{1}{\beta(\beta-1)}$}--(1,0)node[below]{\scriptsize$\frac{1}{\beta-1}$}--(1,1)--(0,1)node[left]{\scriptsize$\frac{1}{\beta-1}$}--(0,.5)--(0,0);
\draw[dotted](.382,0)--(.382,1)(0.618,0)--(0.618,1);
\draw[thick](0,0)--(0.618,1)(.382,0)--(1,1);
\end{tikzpicture}\caption{The projection onto the first coordinate of $K_{\beta}$ for $\beta=\dfrac{1+\sqrt{5}}{2}$}
\end{figure}

This allows the study of all $\beta$-expansions of $x$, each different choice of $\omega$ corresponds to a different $\beta$-expansion of $x$, up to a set of measure zero, and all $\beta$-expansions are given this way. We refer to the $\beta$-expansion $(x_n)_{n=1}^{\infty}$ obtained by iterating $K_{\beta}(\omega,x)$ as the $\beta$-expansion of $x$ generated by $\omega$.

We define the switch region $S:= [\frac{1}{\beta},\frac{1}{\beta(\beta-1)}]$. When the orbit of $(\omega,x)$ enters $\hat S:=\Omega\times S$ we are allowed a choice about how to continue the $\beta$-expansion, this choice is made by looking at the first digit of the sequence $\omega$. We define the hiting number
\[
h(\omega,x,n):=\#\{i \in \{1,\cdots, n\}: K_{\beta}^i(\omega,x) \in \hat S\}.
\]
We note that the dependence of $h(\omega,x,n)$ on the sequence $\omega$ is really just a dependence on the finite word $\omega_1,\cdots,\omega_{h(\omega,x,n)}$, since $\omega$ only influences the orbit of $x$ when $K^n_{\beta}(\omega,x)$ enters the switch region $\hat S$. We see that the $\beta$-expansions of $x$ generated by $\omega$ and $\omega'$ agree to the first $n$ places if and only if $\omega_1,\cdots,\omega_{h(x,\omega,n)}=\omega'_1,\cdots,\omega'_{h(x,\omega',n)}$.

In \cite{DdV}, Dajani and de Vries showed that $K_{\beta}$ has invariant probability measure $\hat{\mu}_{\beta}=\mu_{\beta}\times m$, where $\mu_{\beta}$ is absolutely continuous with respect to $\lambda$ (Lebesgue measure), and $m$ is the $(\frac{1}{2},\frac{1}{2})$ Bernoulli measure on $\{0,1\}^{\mathbb N}$. They also showed that $K_{\beta}$ is ergodic with respect to this measure.
 
We begin by describing $\mathcal N_n(x;\beta)$ in terms of $h(\omega,x,n)$ and $m$.
\begin{lemma}\label{hintegral}
$\mathcal N_n(x;\beta)=\int_{\{0,1\}^{\mathbb N}}2^{h(\omega,x,n)}dm$
\end{lemma}
\begin{proof}
The set $\mathcal E^n_{\beta}(x)$ of $n$-prefixes of $\beta$-expansions of $x$ corresponds to the number of different words $\omega_1,\cdots,\omega_{h(\omega,x,n)}$ for $\omega \in \{0,1\}^{\mathbb N}$. So defining
\[
\Omega(x,n):=\{\omega_1,\cdots,\omega_{h(x,\omega,n)}: \omega \in \{0,1\}^{\mathbb N}\},
\]
we have $|\Omega(x,n)|=\mathcal N_n(x;\beta)$.

We see that $m[\omega_1,\cdots,\omega_{h(\omega,x,n)}]=2^{-h(\omega,x,n)}$. Then
\begin{eqnarray*}
|\Omega(x,n)|&=&\sum_{k=1}^n |\{\omega_1,\cdots,\omega_{h(\omega,x,n)}: \omega \in \{0,1\}^{\mathbb N}, h(\omega,x,n)=k\}|\\
&=&\sum_{k=1}^n 2^{k}m\{\omega \in \{0,1\}^{\mathbb N}: h(\omega,x,n)=k\}\\
&=&\int_{\{0,1\}^{\mathbb N}}2^{h(\omega,x,n)}dm
\end{eqnarray*}
as required.\end{proof}

From this point the proof of Theorem \ref{betathm} is straightforward. By the ergodic theorem we have that for almost every  $(\omega,x)$ (with respect to $\hat{\mu}_{\beta}$) in $[0,\frac{1}{\beta-1}]\times\{0,1\}^{\mathbb N}$,
\[
\lim_{n\rightarrow\infty}\dfrac{h(\omega,x,n)}{n}= \hat{\mu}_{\beta}(\hat S)=\mu_{\beta}(S).
\]

Decomposing $\hat{\mu}_{\beta}=m\times\mu_{\beta}$, and recalling that $\mu_{\beta}$ is absolutely continuous with respect to $\lambda$, we get that for almost every $x$ (w.r.t. $\lambda$), for almost every $\omega$ (w.r.t. $m$),
\[
\lim_{n\rightarrow\infty}\dfrac{h(\omega,x,n)}{n}= \mu_{\beta}(S).
\]
Then, since almost everywhere convergence implies convergence in probability, we have that for almost every $x$ and for all $\epsilon,\delta>0$ there exists a constant $N_{\epsilon \delta}$ such that for all $n>N_{\epsilon \delta}$,
\[
m\left(\{\omega \in \{0,1\}^{\mathbb N} : \left|\dfrac{h(\omega,x,n)}{n}- \mu_{\beta}(S)\right|\geq\epsilon\}\right) <\delta.
\]
We define the good set
\begin{eqnarray*}
G(n,x,\epsilon)&=&\{\omega \in \{0,1\}^{\mathbb N} : \left|\dfrac{h(\omega,x,n)}{n}- \mu_{\beta}(S)\right|<\epsilon\}\\
&=&\{\omega \in \{0,1\}^{\mathbb N} : n(\mu_{\beta}(S)-\epsilon)< h(\omega,x,n)< n(\mu_{\beta}(S)+\epsilon)\}.
\end{eqnarray*}

Now
\[
m(G(n,x,\epsilon)) >1-\delta,
\]
and so for almost every $x$,
\begin{eqnarray*}
\int_{\{0,1\}^{\mathbb N}} 2^{h(\omega,x,n)}dm&\geq&\int_{G(n,x,\epsilon)} 2^{h(\omega,x,n)}dm\\
&\geq& (1-\delta)2^{n(\mu_{\beta}(S)-\epsilon)}.
\end{eqnarray*}
Then

\[
\mathcal N_n(x;\beta) \geq (1-\delta)(2^{n(\mu_{\beta}(S)-\epsilon)}),
\]

and since $\epsilon$ and $\delta$ were arbitrary, we have
\[
\liminf_{n\rightarrow\infty} \dfrac{\log (\mathcal N_n(x;\beta))}{n}\geq \log(2)\mu_{\beta}(S).
\]
This completes the proof of Theorem \ref{betathm}. 

The constant $c(\beta)=\log(2)\mu_{\beta}(S)$ can be computed exactly using using a formula for the density of $\mu_{\beta}$ which will appear in forthcoming work by the author. This formula is reminiscent of the formula of Parry for the invariant density of the greedy $\beta$-transformation, see \cite{Parry}. Unfortunately, the lower bound $\log(2)\mu_{\beta}(S)$ on the growth rate of $\mathcal N_n(x;\beta)$ is not sharp.

\end{section}
\section{Further Questions}\label{Conjectures}
There are some further questions that arise naturally. The first is whether the absolute continuity of $\nu_{\beta}$ is equivalent to the convergence of $f_n$ to a the density of $\nu_{\beta}$. We make the following conjecture.
\begin{conjecture}
If $\nu_{\beta}$ is singular then $\lim_{n\to\infty}f_n(x)=0$ almost everywhere. If $\nu_{\beta}$ is absolutely continuous then $\lim_{n\to\infty}f_n(x)$ exists almost everywhere and is equal to the density of $\nu_{\beta}$.
\end{conjecture}

Bernoulli convolutions are often studied via the sequence of measures $\nu_{\beta,n}$, which converge weakly to $\nu_{\beta}$. One could also define a sequence of measures $m_{\beta,n}$ to be the probability measures with density $f_n$. These measures are the measures obtained by letting $m_{\beta,0}$ be normalised Lebesgue measure on $I_{\beta}$ and by defining
\[
m_{\beta,n+1}=\frac{1}{2}\left(m_{\beta,n}\circ T_0 + m_{\beta,n}\circ T_1\right).
\]
It is easy to see that the sequence $(m_{\beta,n})$ converges weak$^*$ to $\nu_{\beta}$. A proof of the above conjecture would show that if $m_{\beta}$ is absolutely continuous then $(m_{\beta,n})$ converges in the stronger sense that the densities converge almost everywhere to $h_{\beta}$. This would be useful in determining, for example, the multifractal properties of $\nu_{\beta}$.

Our second question is whether one can use the fact that $K_{\beta}$ is a Markov map when $\beta$ is a Salem number to calculate $\mathcal N_n(x;\beta)$. Pisot numbers are the only known examples of $\beta$ for which $\nu_{\beta}$ is singular, but Salem numbers are widely regarded as the most likely candidates for non-Pisot values of $\beta$ which may yield singular Bernoulli convolutions. Feng and Sidorov calculated the growth of $\mathcal N_n(x;\beta)$ for Pisot values of $\beta$, if one were able to extend their methods to include Salem numbers it would allow one to determine whether Salem numbers have absolutely continuous Bernoulli convolutions or not.

Finally, we ask whether it is possible to improve our proof of Theorem \ref{betathm} in order to get precise lower bounds on $\mathcal N_n(x;\beta)$ using the $K_{\beta}$ map. Lemma \ref{hintegral} describes $\mathcal N_n(x;\beta)$ accurately in terms of $h(\omega,x,n)$, but the ergodic theory that we use in our subsequent analysis of $h(\omega,x,n)$ is not strong enough to give a sharp lower bound. Perhaps a more delicate analysis of the ergodic theory of $K_{\beta}$ could yield better understanding of the growth of $\mathcal N_n(x;\beta)$.


\section*{Acknowledgements}
We'd like to thank the many mathematicians with whom we've discussed this work, and in particular Mark Pollicott and Karma Dajani. This work was supported by the Dutch Organisation for Scientific Research (NWO) grant number 613.001.022.


\bibliographystyle{plain} 
\bibliography{betaref.bib}

\end{document}